 \documentclass[final,1p,times]{elsarticle}

\usepackage{amssymb}
 \usepackage{amsthm}
\usepackage{amsmath,amssymb,amsopn,amsfonts,mathrsfs,amsbsy,amscd}
\usepackage{longtable}

\newcommand{\rhdl}{\stackrel{c}{\rhd}}

\newcommand{\prs}{\langle\;,\;\rangle}

\newcommand{\too}{\longrightarrow}
\newcommand{\om}{\omega}
\newcommand{\esp}{\quad\mbox{and}\quad}
\newcommand{\wi}{\widetilde}

\def\br{[\;,\;]}

\newcommand{\G}{\mathfrak{g}}

\newcommand{\h}{{\mathfrak{h}}}
\newcommand{\ad}{{\mathrm{ad}}}
\newcommand{\Ad}{{\mathrm{Ad}}}
\newcommand{\tr}{{\mathrm{tr}}}

\newcommand{\di}{\displaystyle}

\newcommand{\e}{\epsilon}

\newcommand{\de}{\delta}

\newtheorem{Def}{Definition}[section]
\newtheorem{theo}{Theorem}[section]
\newtheorem{pr}{Proposition}[section]

\newtheorem{co}{Corollary}[section]

\newtheorem{remark}{Remark}

\newtheorem{conjecture}{Conjecture}
\font\bb=msbm10

\def\R{\hbox{\bb R}}

\def\N{\hbox{\bb N}}

\def\C{\mathbb{ C}}

\begin{document}

\begin{frontmatter}


 

\title{ Analytic Linear Lie rack Structures on Leibniz Algebras}

\author[label1]{Hamid Abchir}
\address[label1]{Universit\'e Hassan II\\ Ecole Sup\'erieure de Technologie
	\\Route d'El Jadida Km 7, B.P. 8012, 20100 Casablanca, Maroc\\
	{e-mail: h\_abchir@yahoo.com}}
\author[label2]{Fatima-Ezzahrae Abid}
\address[label2]{Universit\'e Cadi-Ayyad\\
	Facult\'e des sciences et techniques\\
	BP 549 Marrakech Maroc\\
	{e-mail: abid.fatimaezzahrae@gmail.com}}
 \author[label3]{Mohamed Boucetta}
 \address[label3]{Universit\'e Cadi-Ayyad\\
 	Facult\'e des sciences et techniques\\
 	BP 549 Marrakech Maroc\\e-mail: m.boucetta@uca.ac.ma}



\begin{abstract}  A linear Lie rack structure on a finite dimensional vector space $V$
	 is a Lie rack operation $(x,y)\mapsto x\rhd y$ pointed at the origin and such that for any $x$, the left translation $\mathrm{L}_x:y\mapsto \mathrm{L}_x(y)= x\rhd y$ is linear. A linear Lie rack operation $\rhd$ is called analytic if for any $x,y\in V$,
	 \[ x\rhd y=y+\sum_{n=1}^\infty A_{n,1}(x,\ldots,x,y), \]where $A_{n,1}:V\times\ldots\times V\too V$ is an $n+1$-multilinear map symmetric in the $n$ first arguments. In this case, $A_{1,1}$ is exactly the left Leibniz product associated to $\rhd$. Any left Leibniz algebra $(\h,\br)$ has a canonical analytic linear Lie rack structure given by $x\rhdl y=\exp(\ad_x)(y)$, where $\ad_x(y)=[x,y]$.
	 
	  In this paper, we show that a sequence $(A_{n,1})_{n\geq1}$ of $n+1$-multilinear maps on a vector space $V$ defines an analytic linear Lie rack structure  if and only if $\br:=A_{1,1}$ is a left Leibniz bracket, the $A_{n,1}$ are invariant for $(V,\br)$ and satisfy a sequence of multilinear equations. Some of these equations have a cohomological interpretation and can be solved when the zero and the 1-cohomology of the left Leibniz algebra $(V,\br)$ are trivial. On the other hand, given a left Leibniz algebra $(\h,\br)$,
	  we  show that there is a large class of (analytic) linear Lie rack structures on $(\h,\br)$ which can be built from the canonical one and invariant multilinear symmetric maps on $\h$. A left Leibniz algebra on which all the analytic linear Lie rack structures are build in this way will be called rigid.   We use our characterizations of analytic linear Lie rack structures to show that $\mathfrak{sl}_2(\R)$ and $\mathfrak{so}(3)$   are rigid.  We conjecture that  any simple Lie algebra is rigid as a left Leibniz algebra.
	
\end{abstract}

\end{frontmatter}








 \section{Introduction} \label{section1}

 In the 1980's, Joyce \cite{joyce} and Matveev \cite{matveev} introduced the notion of  {\it quandle}.
 This notion has been derived from the knot theory, in the way that the axioms of a quandle are the algebraic expressions of Reidemeister moves (I,II,III) for an oriented knot diagram \cite{elhamdadi}. The quandles provide many  knot invariants.  The fundamental quandle or knot quandle was introduced  by Joyce who showed that it is  a complete invariant of a knot (up to a weak equivalence).  Racks which are a generalization of quandles  were introduced by
   Brieskorn \cite{bri} and Fenn and Rourke \cite{fenn}. Recently (see \cite{ carter1, carter2}), there has been investigations on quandles and racks from an algebraic point of view and their relationship with other algebraic structures as  Lie algebras, Leibniz algebras, Frobenius algebras, Yang Baxter equation, and Hopf algebras etc..
  
    A rack is a  non-empty set $\mathit{X}$ together with a map $\rhd  : \mathit{X} \times \mathit{X} \longrightarrow \mathit{X}$, $(a, b) \mapsto a \rhd b$ such that, for any $a,b,c\in \mathit{X}$, the map $\mathrm{L}_{a}  : \mathit{X} \longrightarrow \mathit{X}$, $b \mapsto  a\rhd b$ is a bijection and 
    \begin{equation}\label{r1} a \rhd (b \rhd c) = (a \rhd b)\rhd (a \rhd c).\end{equation}
    
   A rack $\mathit{X}$ is called $\textit{pointed}$ if there exists a distinguished element $e \, \in \mathit{X}$ such that, for any $a\in \mathit{X}$,
  		\begin{equation}\label{r2} a \rhd e = e  \esp   \mathrm{L}_{e}=\mathrm{Id}_X.\end{equation}
  		 A rack $\mathit{X}$ is called a quandle if, for any $a\in \mathit{X}$, $a \rhd a=a$.
  		 
  		 A Lie rack is a rack $(X,\rhd)$ such that $X$ is a smooth manifold, $\rhd$ is a smooth map and the left translations $\mathrm{L}_{a}$ are diffeomorphisms. 
  		 Any Lie group $G$ has a Lie rack structure given by $g\rhd h=g^{-1}hg$.

 {\it Leibniz algebras} were first introduced and investigated in the papers of  Bloh \cite{bloh, bloh1} under the name of D-algebras.
  Then they were  rediscovered by  Loday \cite{loday} who called them  Leibniz algebras. A  left Leibniz algebra is an algebra $(\h,\br)$ over a field $\mathbb{K}$ such that, for every element $u\in\h$, $\ad_u:\h\too\h$, $v\mapsto[u,v]$ is a derivation of $\mathfrak{h}$, i.e., 
  \begin{equation}\label{l1} [u,[v,w]]=[[u,v],w]+[v,[u,w]],\quad v,w\in\h. \end{equation}Any Lie algebra is a left Leibniz algebra and a left Leibniz algebra is a Lie algebra if and only if its bracket is skew-symmetric.
  Many results of the theory of Lie algebras can be extended to left Leibniz algebras (see \cite{ayp, ayp1, ayp2}). 
 
  In 2004,  Kinyon \cite{kinyon} proved that if $(X,e)$ is a pointed Lie rack, $T_eX$  carries a structure of left Leibniz algebra. 
  Moreover,  in the case when  the Lie rack structure is associated to a Lie group $G$ then the associated left Leibniz algebra is the Lie algebra of $G$. 
  
  Given a pointed Lie rack $(X,e)$, for any $a\in X$, we denote by $\Ad_a:T_eX=\h\too \h$ the differential of $\mathrm{L}_a$ at $e$. We have
  \[ \mathrm{L}_{a\rhd b}=\mathrm{L}_a\circ \mathrm{L}_b\circ \mathrm{L}_a^{-1}\esp \Ad_{a\rhd b}=
  \Ad_a\circ \Ad_b\circ \Ad_a^{-1}. \]
  Thus $\Ad:X\too \mathrm{GL}(\h)$ is an homomorphism of Lie racks. If we put
  \[ [u,v]_{\rhd}=\frac{d}{dt}_{|t=0}\Ad_{c(t)}v,\quad u,v\in\h,c:]-\e,\e[\too X,\; c(0)=e,c'(0)=u, \]$(\h,[\;,\;]_{\rhd})$ becomes a left Leibniz algebra.

  A {\it linear} Lie rack structure on a finite dimensional vector space $V$
  is a Lie rack operation $(x,y)\mapsto x\rhd y$ pointed at $0$ and such that for any $x$, the map $\mathrm{L}_x:y\mapsto x\rhd y$ is linear. A linear Lie rack operation $\rhd$ is called {\it analytic} if for any $x,y\in V$,
  \begin{equation}\label{eq4}
  x\rhd y=y+\sum_{n=1}^\infty A_{n,1}(x,\ldots,x,y),
  \end{equation}where for each $n$, $A_{n,1}:V\times\ldots\times V\too V$ is an $n+1$-multilinear map which is symmetric in the $n$ first arguments. In this case, $A_{1,1}$ is the left Leibniz bracket associated to $\rhd$.

  If $(\h,\br)$ is a left Leibniz algebra then the operation $\stackrel{c}{\rhd}:\h\times\h\too\h$ given by $$u\rhdl v=\exp(\ad_u)(v)$$ defines an analytic linear Lie rack   structure on $\h$  such that the associated left Leibniz bracket on $T_0\h=\h$ is the initial bracket $\br$.
 We call $\rhdl$ the {\it canonical} linear Lie rack structure associated to $(\h,\br)$. 
 
 In this paper, we will study linear Lie rack structures with an emphasis on analytic linear Lie rack structures. 
 
 Actually, there is a large class of linear Lie rack structures on $(\h,\br)$ containing the canonical one. This class was suggested to us by an example sent to us by Martin Bordemann.  The proof of the following proposition will be given in Section \ref{section2}.
 
 \begin{pr}\label{pr1} Let $(\h,\br)$ be a left Leibniz algebra, $F:\R\too\R$ a smooth function and $P:\h\times\ldots\times\h\too\R$ a symmetric multilinear $p$-form such that, for any $y,x_1\ldots,x_p\in\h$,
 	\[ \sum_{i=1}^pP(x_1,\ldots,[y,x_i],\ldots,x_p)=0. \]Then the operation $\rhd$ given by
 	\[ x\rhd y=\exp(F(P(x,\ldots,x))\ad_x)(y) \]is a linear Lie rack structure on $\h$ and its associated left Leibniz bracket is $[\;,\;]_{\rhd}=F(0)\br$. Moreover, if $F$ is analytic then $\rhd$ is analytic .
 	
 \end{pr}
 
 This proposition shows that  a left Leibniz algebra might be associated to many non equivalent pointed Lie rack structures. For instance if one takes $F(0)=0$ in Proposition \ref{pr1}, the two pointed Lie rack structures
 \[ x\rhd_0 y=y\esp x\rhd_1 y=\exp(F(P(x,\ldots,x)))\ad_x)(y)  \]are two pointed Lie rack structures on $\h$ which are not equivalent (even locally near 0) and have the same left Leibniz algebra, namely, the abelian one. This contrasts with the theory of Lie groups where two Lie groups are locally equivalent near their unit elements if and only if they have the same Lie algebra. Moreover, this proposition motivates the study of linear Lie rack structures and gives
  a sense to the following definition.
 
 \begin{Def}\label{def} A left Leibniz algebra $(\h,\br)$ is called rigid if any 
 analytic	linear Lie rack structure $\rhd$ on $\h$ such that $\br_{\rhd}=\br$ is given by
 	\[ x\rhd y=\exp(F(P(x,\ldots,x)))\ad_x)(y), \]where $F:\R\too\R$ is analytic with $F(0)=1$ and  $P:\h\times\ldots\times\h\too\R$ is a symmetric multilinear $p$-form such that, for any $y,x_1\ldots,x_p\in\h$,
 	\[ \sum_{i=1}^pP(x_1,\ldots,[y,x_i],\ldots,x_p)=0. \]

 \end{Def}

\begin{remark} We have seen that the abelian left Leibniz algebra is not rigid. \end{remark}

  This paper is an introduction to the study of the rigidity of left Leibniz algebras.   Our approach was suggested to us by the one used in the study of linearization of Poisson structures (see \cite{conn}).  One of our main results is the following theorem.
  \begin{theo}\label{main} Let $V$ be a real finite dimensional vector space and $(A_{n,1})_{n\geq1}$ a sequence of $n+1$-multilinear maps symmetric in the $n$ first arguments. We suppose that the operation  $\rhd$ given by 
  	\[ x\rhd y=y+\sum_{n=1}^{\infty}A_{n,1}(x,\ldots,x,y) \]converges.
  	Then
  	  $\rhd$  is a  Lie rack structure on $V$  if and only if   for any $p,q\in\N^*$ and $x,y,z\in V$,
  	\begin{eqnarray}\label{eqm}
  		A_{p,1}(x,A_{q,1}(y,z))&=&\sum_{s_1+\ldots+s_q+k=p}A_{q,1}(A_{s_1,1}(x,y),\ldots,A_{s_q,1}(x,y),A_{k,1}(x,z)),
  	\end{eqnarray}where for sake of simplicity $A_{p,1}(x,y):=A_{p,1}(x,\ldots,x,y)$. 
  	
  	In particular, if $p=q=1$ we get that $\br:=A_{1,1}$ is a left Leibniz bracket which is actually the left Leibniz bracket associated to $(V,\rhd)$.
  	
  \end{theo}
  \begin{remark}
  When $p=1$ and $q\in\N^*$, the relation \eqref{eqm} becomes
  \begin{equation}\label{inv} \mathcal{L}_xA_{q,1}(y_1,\ldots,y_{q+1}):=[x,A_{q,1}(y_1,\ldots,y_{q+1})]
  -\sum_{i=1}^{q+1}A_{q,1}(y_1,\ldots,[x,y_i],\ldots,y_{q+1})=0. \end{equation}
  \end{remark}
 A multilinear map on a left Leibniz algebra satisfying \eqref{inv} will be called {\it invariant}. Thus Theorem \ref{main} reduces the study of analytic linear Lie rack structures to the study of the datum of a left Leibniz algebra with a sequence of invariant multilinear maps satisfying a sequence of multilinear equations. Even though equations \eqref{eqm} are complicated, we will see in this paper that they are far more easy to handle than the distributivity condition \eqref{r1}. In Section \ref{section2} we will give the proofs of Proposition \ref{pr1} and Theorem \ref{main} and we will show that there is a large class of non rigid left Leibniz algebras (see Corollary \ref{co}). On the other hand, when $q=1$ the equation \eqref{eqm} has a cohomological interpretation with respect to the cohomology of the left Leibniz algebra $(V,\br)$.  When $H^0=H^1=0$ we can deduce a refined expression of the $(A_{n,1})$ (see Theorem \ref{main2} in Section \ref{section4}). By using Theorems \ref{main} and \ref{main2} we will prove that  $\mathfrak{sl}_2(\R)$ and $\mathfrak{so}(3)$  are rigid  (see Sections \ref{section5}).  As the reader will see, the proof of the rigidity of $\mathfrak{sl}_2(\R)$ and $\mathfrak{so}(3)$ based on Theorems \ref{main} and \ref{main2} is quite difficult and has needed a deep understanding of the structure of these Lie algebras as a simple Lie algebras. We think that the study of the following conjecture can be a challenging mathematical problem.
 \begin{conjecture} Every simple Lie algebra is rigid in the sense of Definition \ref{def}.
 	
 \end{conjecture}

 \section{Some classes of non rigid left Leibniz algebras,  proofs of Proposition \ref{pr1} and Theorem \ref{main}} \label{section2}

The proof of Proposition \ref{pr1} is a consequence of the following well-known result.
\begin{pr} \label{pr2} Let $(X,\rhd)$ be a rack and $J:X\too X$ a map such that, for any $x,y\in X$, $J(x\rhd y)=x\rhd J(y)$, i.e., $J\circ 
	\mathrm{L}_x=\mathrm{L}_x\circ J$ for any $x$. Then the operation
	\[ x\rhd_Jy=J(x)\rhd y \]defines a rack structure on $X$.
	
\end{pr}

\begin{proof} We have, for any $x,y,z\in X$,
	\begin{eqnarray*}
	x\rhd_J(y\rhd_J z)&=&J(x)\rhd(J(y)\rhd z)\\
	&=&(J(x)\rhd J(y))\rhd (J(x)\rhd z)\\
	&=&(J(J(x)\rhd y))\rhd (x\rhd_Jz)\\
	&=&(x\rhd_Jy)\rhd_J(x\rhd_Jz).\qedhere
	\end{eqnarray*}
	
\end{proof}

\noindent{\bf Proof of Proposition \ref{pr1}.}
\begin{proof} We consider the map $J:\h\too\h$ given by $J(x)=F(P(x,\ldots,x))x$. Since $P$ is invariant, we have $P(\exp(\ad_x)(y),\ldots,\exp(\ad_x)(y))=P(y,\ldots,y)$ and hence $J(x\rhdl y)=x\rhdl J(y)$ and one can apply Proposition \ref{pr2} to conclude.
	\end{proof}

 The following proposition shows that the class of non rigid left Leibniz algebras is large. Recall that if $\h$ is a left Leibniz algebra then its center $Z(\h)=\{a\in\h,[a,\h]=[\h,a]=0  \}$.
 
 \begin{pr}\label{pr22}
 	Let $(\h,\br)$ be a left Leibniz algebra. Choose a scalar product $\prs$ on $\h$,  $(a_1,b_1,\ldots,a_k,b_k)$ a family of vectors in $[\h,\h]^\perp\cap Z(\h)^\perp$,  $(z_1,\ldots,z_k)$ a family of vector in $ Z(\h)$ and $f_1,\ldots,f_k:\R\too\R$ with $f_j(0)=0$ for $j=1,\ldots,k$.  If $\rhdl$ is the canonical linear Lie rack operation on $\h$ then
 	\[ x\rhd y=x\rhdl y+\sum_{j=1}^k\langle y,b_j\rangle f_j(\langle x,a_j\rangle) z_j \] is a linear Lie rack operation pointed at 0. Moreover, if   $f'_j(0)=0$ for $j=1,\ldots,k$ then $\br_{\rhd}=\br$.\end{pr}
 	\begin{proof} Note first that for any $z\in Z(\h)$ and for any $x\in\h$, $x\rhd z=z$ and $z\rhd x=x$. Moreover, $\langle b_j,x\rhd y\rangle=\langle b_j,x\rhdl y\rangle=\langle b_j,y\rangle$ and $(x\rhd y)\rhdl z=(x\rhdl y)\rhdl z$ for any $x,y,z\in\h$. So, for any $x,y,z\in\h$,
 		\begin{eqnarray*}
 		x\rhd(y\rhd z)&=&x\rhd(y\rhdl z)+\sum_{j=1}^k\langle z,b_j\rangle f_j(\langle y,a_j\rangle) x\rhd z_j\\
 		&=&x\rhdl(y\rhdl z)+\sum_{j=1}^k\langle y \rhdl z,b_j\rangle f_j(\langle x,a_j\rangle) z_j+\sum_{j=1}^k\langle z,b_j\rangle f_j(\langle y,a_j\rangle)  z_j,\\
 		&=&x\rhdl(y\rhdl z)+\sum_{j=1}^k\langle  z,b_j\rangle f_j(\langle x,a_j\rangle) z_j+\sum_{j=1}^k\langle z,b_j\rangle f_j(\langle y,a_j\rangle)  z_j,\\
 		(x\rhd y)\rhd(x\rhd z)&=&(x\rhd y)\rhd(x\rhdl z)+\sum_{j=1}^k\langle z,b_j\rangle f_j(\langle x,a_j\rangle) z_j\\
 		&=&(x\rhd y)\rhdl(x\rhdl z)+\sum_{j=1}^k\langle x\rhdl z,b_j\rangle f_j(\langle x\rhd y,a_j\rangle) z_j+\sum_{j=1}^k\langle z,b_j\rangle f_j(\langle x,a_j\rangle) z_j\\
 		&=&(x\rhdl y)\rhdl(x\rhdl z)+\sum_{j=1}^k\langle  z,b_j\rangle f_j(\langle  y,a_j\rangle) z_j+\sum_{j=1}^k\langle z,b_j\rangle f_j(\langle x,a_j\rangle) z_j.
 		\end{eqnarray*}This proves the proposition.
 		\end{proof}

 \begin{co}\label{co}
 	\begin{enumerate}
 		\item Let $\h$ be a left Leibniz algebra which is a Lie algebra such that $[\h,\h]+ Z(\h)\not= \h$,  $Z(\h)\not=\{0\}$. Then $\h$ is not rigid.
 		
 		\item Let $\h$ be a left Leibniz algebra  such that $[\h,\h]+ Z(\h)\not= \h$ and $Z(\h)$ is not contained in $[\h,\h]$. Then $\h$ is not rigid.
 	\end{enumerate} 
 	
 \end{co}
 
 \begin{proof}\begin{enumerate}
 		\item  By virtue of Definition \ref{def}, if $\h$ is rigid then any linear analytic rack structure $\rhd$ on $\h$ satisfies $x\rhd x=x$ for any $x\in\h$. Choose $z\in Z(\h)\setminus\{0\}$, a scalar product $\prs$ on $\h$ and $a\in[\h,\h]^\perp\cap Z(\h)^\perp$ with $a\not=0$. According to Proposition \ref{pr22}, the operation
 	\[ x\rhd y=x\rhdl y+\langle x,a\rangle^2\langle y,a\rangle z \]is an analytic linear Lie rack structure on $\h$ satisfying $\br_\rhd=\br$. However, this operation satisfies $a\rhd a=a+|a|^6z\not=a$ and hence $\h$ is not rigid. 	
 	
 	\item We have also that if $\h$ is rigid then    any linear analytic rack structure $\rhd$ on $\h$ satisfies $x\rhd y-x\rhdl y\in[\h,\h]$. We proceed as the first case and we consider the same Lie rack operation on $\h$  with $a\in Z(\h)$ and $a\notin[\h,\h]$ and we get a contradiction.\qedhere
 \end{enumerate}
 		\end{proof}
 		
 		\begin{remark} There is a large class of left Leibniz algebras satisfying the hypothesis of Corollary  \ref{co}, for instance, any 2-step nilpotent Lie algebra belongs to this class. 
 			
 		\end{remark}
 
 \noindent{\bf Proof of Theorem \ref{main}.} \begin{proof} Put $A_{0,1}(x,y)=y$.
 	We have
 	\begin{eqnarray*}
 		x\rhd(y\rhd z)&=&\sum_{n\in\N} A_{n,1}(x,\ldots,x,y\rhd z)\\
 		&=&\sum_{n,p\in\N} A_{n,1}(x,\ldots,x,A_{p,1}(y,\ldots,y,z)),
 	\end{eqnarray*}
 	\begin{eqnarray*}
 		(x\rhd y)\rhd(x\rhd z)&=&\sum_{n=0}^\infty A_{n,1}(x\rhd y,\ldots,x\rhd y,x\rhd z)\\
 		&=&\sum_{n,s_1,\ldots,s_n,k} A_{n,1}(A_{s_1,1}(x,y),\ldots,A_{s_n,1}(x,y),A_{k,1}(x,z)).
 	\end{eqnarray*}By identifying the homogeneous component of degree $n$ in $x$ and of degree $p$ in $y$ in both $x\rhd(y\rhd z)$ and $(x\rhd y)\rhd(x\rhd z)$ we get the desired relation.
 	\end{proof}
 	
 	The following result is an immediate and important consequence of Theorem \ref{main}. 
 	\begin{co}\label{co1} Let $(\h,\br)$ be a left Leibniz algebra and $\rhdl$ its canonical linear Lie rack operation. Then
 		\[ x\rhdl y=\sum_{n=0}^\infty A_{n,1}^0(x,\ldots,x,y) \]where
 		\[ A_{0,1}^0(x,y)=y\esp  A_{n,1}^0(x_1,\ldots,x_n,y)=\frac1{(n!)^2}\sum_{\sigma\in S_n}\ad_{x_{\sigma(1)}}\circ\ldots \circ\ad_{x_{\sigma(n)}}(y), \]and $S_n$ is the group of permutations of $\{1,\ldots,n\}$.
 		Furthermore, the $\left(A_{n,1}^0\right)_{n\in\N}$ satisfy the sequence of equations \eqref{eqm}.
 	\end{co}

 \section{ Analytic linear Lie racks structures over left Leibniz algebras with trivial 0-cohomology and 1-cohomology } \label{section4}
 
 In this section, we recall the definition of the cohomology of a left Leibniz algebra. We will give an important expression of the $A_{n,1}$ defining an analytic linear Lie rack structure on a left Leibniz algebra $\h$ when $H^0(\h)=H^1(\h)=0$.

 Let $(\h,\br)$ be a left Leibniz algebra. For any $n\geq0$, the operator
 $\de:Hom(\otimes^n\h,\h)\too Hom(\otimes^{n+1}\h,\h)$  given by
 	\begin{eqnarray*} \de(\om)(x_0,\ldots,x_n)&=&\sum_{i=0}^{n-1}[x_i,\om(x_0,\ldots,\hat{x_i},\ldots,x_n)]
 	+(-1)^{n-1}[\om(x_0,\ldots,x_{n-1}),x_n]\\&&+\sum_{i<j}(-1)^{i+1}\om(x_0,\ldots,\hat{x}_i,\ldots,x_{j-1},[x_i,x_j],
 	x_{j+1},\ldots,x_n), \end{eqnarray*} satisfies $\de^2=0$ and then defines a cohomology $H^p(\h)$ for $p\in\N$. For any $x\in\h$ and $F,G:\h\too\h$, we have
 \[ \de(x)(m)=-[x,m] \esp \de(F)(y,z)=[y,F(z)]+[F(y),z]-F([y,z])\]and one can see easily that
 \begin{equation}\label{comp} \de(F\circ G)(y,z)=\de(F)(y,G(z))+\de(F)(G(y),z)+F\circ\de(G)(y,z)-[F(y), G(z)]-[G(y),F(z)]. \end{equation}

 \begin{remark}\label{remark} Let $(\h,\br)$ be a left Leibniz algebra which is a Lie algebra. The cohomology of $\h$ as left Leibniz algebra is different from its cohomology as a Lie algebra, however  $H^0$ and $H^1$ are the same for both  cohomologies.
 	
 \end{remark}
 
 Now let's take a closer look to equations \eqref{eqm} when $q=1$. 
 Let $(\h,\br)$ be a left Leibniz algebra and $(A_{n,1})_{p\in\N}$ a sequence of  $(n+1)$-multilinear maps on $\h$ with  values in $\h$  symmetric in the $n$ first arguments and such that  $A_{0,1}(x,y)=y$ and $A_{1,1}(x,y)=[x,y]$ . For  sake of simplicity we write $A_{n,1}(x,y)=A_{n,1}(x,\ldots,x,y)$.
 
 Equation \eqref{eqm} for $q=1$ can be written for any $x,y,z\in\h$,
 \[ A_{p,1}(x,[y,z])=[y,A_{p,1}(x,z)]+[A_{p,1}(x,y),z]+\sum_{r=1}^{p-1}[A_{r,1}(x,y),A_{p-r,1}(x,z)]. \]Thus
 \begin{equation}\label{eqc}
 \de(i_x\ldots i_xA_{p,1})(y,z)=-\sum_{r=1}^{p-1}[A_{r,1}(x,y),A_{p-r,1}(x,z)],
 \end{equation}where $i_x\ldots i_xA_{p,1}:\h\too\h$, $y\mapsto A_{p,1}(x,\ldots,x,y)$.
 
 On the other hand, the sequence $(A_{n,1}^0)_{n\in\N}$ defining the canonical linear Lie rack structure of $\h$ (see Corollary \ref{co1}) satisfies \eqref{eqm} and hence
 \begin{equation}\label{eqc1}
 \de(i_x\ldots i_xA_{p,1}^0)(y,z)=-\sum_{r=1}^{p-1}[A_{r,1}^0(x,y),A_{p-r,1}^0(x,z)].
 \end{equation}If  $p=2$,  since $A_{0,1}=A_{0,1}^0$ and $A_{1,1}=A_{1,1}^0$,  Equations \eqref{eqc} and \eqref{eqc1} implies that, for any $x\in\h$,
 \[ \de(i_xi_xA_{2,1}-i_xi_xA_{2,1}^0)=0. \]Since $A_{2,1}$ and $A_{2,1}^0$ are symmetric in the two first arguments this is equivalent to
 \[ \de(i_xi_yA_{2,1}-i_xi_yA_{2,1}^0)=0,\quad\mbox{for any}\;x,y\in\h. \]
  This is a cohomological equation and if $H^1(\h)=0$ then there exists $B_2:\h\times\h\too\h$  such that, for any $x,y,z\in\h$,
 \begin{equation} \label{h1} A_{2,1}(x,y,z)=A_{2,1}^0(x,y,z)+[B_2(x,y),z].\end{equation} Moreover, if $H^0(\h)=0$ then $B_2$ is unique and symmetric and one can check easily that $A_{2,1}$ is invariant if and only if $B_2$ is invariant.

 We have  triggered an induction process and, under the same hypothesis, the $(A_{p,1})_{p\geq2}$ satisfy a similar formula as \eqref{h1}. This is the purpose of the following theorem.

 \begin{theo}\label{main2} Let $(\h,[\;,\;])$ be a left Leibniz algebra such that $H^0(\h)=H^1(\h)=0$. Let $(A_{n,1})_{n\geq0}$ be a sequence where $A_{0,1}(x,y)=y$ and $A_{1,1}(x,y)=[x,y]$ and, for any $n\geq2$, $A_{n,1}:\h\times\ldots\times \h\too\h$ is multilinear invariant and symmetric in the $n$ first arguments.  We suppose that the $A_{n,1}$ satisfy \eqref{eqc}. Then
 	there exists a unique sequence $(B_n)_{n\geq2}$ of invariant symmetric multilinear maps $B_n:\h\times\ldots\times\h\too\h$ such that, for any $x,y\in\h$,
 	\begin{equation}\label{eq} A_{n,1}(x,y)=A_{n,1}^0(x,y)+\sum_{\begin{array}{c}
 		1\leq k\leq\left[\frac{n}2\right]\\  s=l_1+\ldots+l_{k}\leq n
 		\end{array}	} A_{k,1}^0(B_{l_1}(x),\ldots,B_{l_k}(x),A_{n-s,1}^0(x,y)), \end{equation}where $A_{p,1}(x,y)=A_{p,1}(x,\ldots,x,y)$ and $B_l(x)=B_l(x,\ldots,x)$.
 	
 \end{theo}

 \begin{remark} Formula \eqref{eq} deserves some explications. As any formula depending inductively on $n$, to find the general form one needs to check it for the first values of $n$ and it is what we have done. There are the formulas we found directly and which helped us to establish the expression \eqref{eq}.
 	\begin{eqnarray*}
 		A_{3,1}(x,y)&=&
 		A_{3,1}^0(x,y)+[B_2(x),A_{1,1}^0(x,y)]+[B_3(x),A_{0,1}^0(x,y)]\\
 		&=&A_{3,1}^0(x,y)+A_{1,1}^0(B_2(x),A_{1,1}^0(x,y))+A_{1,1}^0(B_3(x),A_{0,1}^0(x,y))
 		,\\
 		A_{4,1}(x,y)&=&A_{4,1}^0(x,y)+[B_4(x),A_{0,1}^0(x,y)]+[B_3(x),A_{1,1}^0(x,y)]+
 		[B_2(x),A_{2,1}^0(x,y)]\\&&+\frac12[B_2(x),[B_2(x),A_{0,1}^0(x,y)]],\\
 		A_{5,1}(x,y)&=&A_{5,1}^0(x,y)+[B_5(x),y]+[B_4(x),A_{1,1}^0(x,y)]+[B_3(x),A_{2,1}^0(x,y)]+
 		[B_2(x),A_{3,1}^0(x,y)]\\
 		&&\frac12\left( [B_2(x),[B_3(x),y]]+[B_3(x),[B_2(x),y]]   \right)+\frac12[B_2(x),[B_2(x),A_{1,1}^0(x,y)]].
 	\end{eqnarray*}

 \end{remark}

 To prove Theorem \ref{main2}, we will proceed by induction. The proof is rather technical and needs some preliminary formulas.

 Fix $n\geq2$ and $x\in\h$. For any $1\leq k\leq\left[ \frac{n+1}2 \right]$ and
 $s=l_1+\ldots+l_k\leq n+1$, in the proof of Theorem \ref{main2}, we will need to compute  $\de(F_k\circ G_{s})$ where
  $F_{k},G_{s}:\h\too\h$ are given by
 \[ F_{k}(y)= A_{k,1}^0(B_{l_1}(x),\ldots,B_{l_k}(x),y),\quad G_{s}(y)=A_{n+1-s,1}^0(x,y).  \]
 This is straightforward from \eqref{comp} and the formula
 \[ \de(i_{x_1}\ldots i_{x_k}A_{k,1}^0)(y,z)=-\frac1{k!}\sum_{p=1}^{k-1}
 \sum_{\sigma\in S_k}[A_{p,1}^0(x_{\sigma(1)},\ldots,x_{\sigma(p)},y),
 A_{k-p,1}^0(x_{\sigma(p+1)},\ldots,x_{\sigma(k)},z)] \]whose polar form is \eqref{eqc1}. We use here the well-know fact that two symmetric multilinear forms are equal if and only if their polar forms are equal. For sake of simplicity put
  $Q(k,s)=\de(F_k\circ G_{s})(y,z)$.

 \begin{pr}\label{prQ} We have
 	{\small
 		\begin{eqnarray*}  
 			Q(1,s)&=&-\sum_{r=0}^{n-s}[A_{1,1}^0(B_s(x),A_{r,1}^0(x,y)),A_{n+1-s-r,1}^0(x,z)]
 			-\sum_{r=1}^{n+1-s}[A_{r,1}^0(x,y),A_{1,1}^0(B_s(x),A_{n+1-s-r,1}^0(x,z))],
 			\quad s\leq n,\\
 			Q(k,n+1)
 			&=&-\frac1{k!}\sum_{h=1}^{k-1}
 			\sum_{\sigma\in S_k}[A_{h,1}^0(B_{l_{\sigma(1)}}(x),\ldots,B_{l_{\sigma(h)}}(x),y),
 			A_{k-h,1}^0(B_{l_{\sigma(h+1)}}(x),\ldots,B_{l_{\sigma(k)}}(x),z)],
 			\quad k\geq2,\\
 			Q(k,s)
 			&=&-\frac1{k!}\sum_{r=0}^{n+1-s}
 			\sum_{p=1}^{k-1}\sum_{\sigma\in S_k}
 			[A_{p,1}^0(B_{l_{\sigma(1)}}(x),\ldots,B_{l_{\sigma(p)}}(x),A_{r,1}^0(x,y)),
 			A_{k-p,1}^0(B_{l_{\sigma(p+1)}}(x),\ldots,B_{l_{\sigma(k)}}(x),A_{n+1-s-r,1}^0(x,z))],\\
 			&&-\sum_{r=1}^{n+1-s}[A_{r,1}^0(x,y),
 			A_{k,1}^0(B_{l_{1}}(x),\ldots,B_{l_{k}}(x),A_{n+1-s-r,1}^0(x,z))]
 			-\sum_{r=0}^{n-s}
 			[A_{k,1}^0(B_{l_{1}}(x),\ldots,B_{l_{k}}(x),A_{r,1}^0(x,y)),
 			A_{n+1-s-r,1}^0(x,z)],\\&&\quad \quad k\geq2,s\leq n.
 		\end{eqnarray*}}
 	\end{pr}

\noindent {\bf Proof of Theorem \ref{main2}.}
 \begin{proof} We prove the formula by induction on $n$. For $n=2$, the formula has been established in \eqref{h1}.
 	
 	Suppose that there exists a family  $(B_2,\ldots,B_n)$ where $B_k$ is an invariant symmetric $k$-from on $\h$ with  values in $\h$  such that for any $2\leq r\leq n$,
 	\begin{eqnarray}\label{ih} A_{r,1}(x,y)&=&A_{r,1}^0(x,y)+\sum_{\begin{array}{c}
 		1\leq k\leq\left[\frac{r}2\right]\\  l_1+\ldots+l_{k}=s\leq r
 		\end{array}	} A_{k,1}^0(B_{l_1}(x),\ldots,B_{l_k}(x),A_{r-s,1}^0(x,y)). \end{eqnarray}	
 	We look for $B_{n+1}:\h\times\ldots\times \h\too\h$ symmetric and invariant such that
 	\begin{eqnarray*} A_{n+1,1}(x,y)&=&A_{n+1,1}^0(x,y)+\sum_{\begin{array}{c}1\leq k\leq\left[\frac{n+1}2\right]\\
 		  l_1+\ldots+l_{k}=s\leq n+1
 		\end{array}	} A_{k,1}^0(B_{l_1}(x),\ldots,B_{l_k}(x),A_{n+1-s,1}^0(x,y))\\
 		&=&[B_{n+1}(x),y]+A_{n+1,1}^0(x,y)+\sum_{\begin{array}{c}1\leq k\leq\left[\frac{n+1}2\right]\\
 				l_1+\ldots+l_{k}=s\leq n+1\\l_1,\ldots,l_k\leq n
 			\end{array}	} A_{k,1}^0(B_{l_1}(x),\ldots,B_{l_k}(x),A_{n+1-s,1}^0(x,y))\\
 			&=&[B_{n+1}(x),y]+R(x)(y),
 		\end{eqnarray*}where $R(x)$ depends only on $(B_2,\ldots,B_n)$.
 	
 	The idea of the proof is to show that, for any $x\in\h$, $\de(D(x))=0$ where $D(x):\h\too\h$ is given by $D(x)(y)=A_{n+1,1}(x,y)-R(x)(y)$.
 	Then since $H^0(\h)=H^1(\h)=0$ there exists a unique $B_{n+1}$ satisfying $D(x)(y)=[B_{n+1}(x),y]$. By using the fact that $D(x)$ is the polar form of a symmetric form and $H^0(\h)=0$ one can see that $B_{n+1}(x)$ is the polar form of a symmetric form which is also invariant.

 	Let us compute now $\de(D(x))$.	According to  \eqref{eqc}, we have
 		\begin{eqnarray*}
 			\de(i_{x}\ldots i_{x}A_{n+1,1})(y,z)&=&-\sum_{r=1}^n[A_{r,1}(x,\ldots,x,y),A_{n+1-r,1}(x,\ldots,x,z)].
 		\end{eqnarray*} By expanding this relation using our induction hypothesis given in \eqref{ih}, we get that
 		\[ \de(i_{x}\ldots i_{x}A_{n+1,1})(y,z)=\de(i_{x}\ldots i_{x}A_{n+1,1}^0)(y,z)+S+T+U, \]where
 		
 		{\footnotesize	\begin{eqnarray*}
 				S
 	&=&-\sum_{r=1}^{n-1}\sum_{\begin{array}{c}1\leq k\leq\left[\frac{n+1-r}2\right]\\
 			  s=l_1+\ldots+l_{k}\leq n+1-r
 		\end{array}}[A_{r,1}^0(x,\ldots,x,y),A_{k,1}^0(B_{l_1}(x),\ldots,B_{l_k}(x),A_{n+1-r-s,1}^0(x,z))],\\
 T&=&-\sum_{r=2}^n\sum_{\begin{array}{c}1\leq k\leq\left[\frac{r}2\right]\\
 	  s=l_1+\ldots+l_{k}\leq r
 	\end{array}	}[A_{k,1}^0(B_{l_1}(x),\ldots,B_{l_k}(x),A_{r-s,1}^0(x,y)),A_{n+1-r,1}^0(x,\ldots,x,z)],\\
 U&=&-\sum_{r=2}^{n-1}\sum_{\begin{array}{c}
 		1\leq k\leq\left[\frac{r}2\right]\\  s_1=l_1+\ldots+l_{k}\leq r\end{array}	}
 	\sum_{\begin{array}{c}1\leq h\leq\left[\frac{n+1-r}2\right]\\  s_2=p_1+\ldots+p_{h}\leq n+1-r
	\end{array}	}\\&&[A_{k,1}^0(B_{l_1}(x),\ldots,B_{l_k}(x),A_{r-s_1,1}^0(x,y)),A_{h,1}^0(B_{p_1}(x),\ldots,B_{p_h}(x),A_{n+1-r-s_2,1}^0(x,z))].		\end{eqnarray*}}
 On the other hand, if we denote
 	\[ D_{k,s}(x)(y)=A_{k,1}^0(B_{l_1}(x),\ldots,B_{l_k}(x),A_{n+1-s,1}^0(x,y)), \]
 	we remark that the computation of $\de(R(x))$ is based on Proposition \ref{prQ} where we have computed  the $\de(D_{k,s}(x))$. 
 	
 	To conclude, we need to show that
 	\[ S+T+U=\sum_{\begin{array}{c}1\leq k\leq\left[\frac{n+1}2\right]\\
 	  l_1+\ldots+l_{k}=s\leq n+1\\ l_1,\ldots,l_k\leq n
 	\end{array}	}Q(k,s), \]	where $Q(k,s)$ is given in Proposition \ref{prQ}. Let $S_1$ and $T_1$ be the terms in $S$ and $T$ corresponding to $k=1$. We have
 \begin{eqnarray*}
 	S_1&=&-\sum_{r=1}^{n-1}\sum_{\begin{array}{c} 2\leq s\leq n+1-r
	\end{array}	}[A_{r,1}^0(x,\ldots,x,y),
 A_{1,1}^0(B_{s}(x),A_{n+1-r-s,1}^0(x,z))],\\
	T_1&=&-\sum_{r=2}^n\sum_{\begin{array}{c} 2\leq s\leq r
 	\end{array}	}[A_{1,1}^0(B_{s}(x),A_{r-s,1}^0(x,y)),A_{n+1-r,1}^0(x,\ldots,x,z)].
 	\end{eqnarray*}On the other hand,
 	\begin{eqnarray*} 
 	\sum_{2\leq s\leq n}Q(1,s)&=&-\sum_{2\leq s\leq n}\sum_{r=0}^{n-s}[A_{1,1}^0(B_s(x),A_{r,1}^0(x,y)),A_{n+1-s-r,1}^0(x,z)]
 	-\sum_{2\leq s\leq n}\sum_{r=1}^{n+1-s}[A_{r,1}^0(x,y),A_{1,1}^0(B_s(x),A_{n+1-s-r,1}^0(x,z))].	
 	\end{eqnarray*}	Since
 	\begin{eqnarray*} \{(r,s), 1\leq r\leq n-1,2\leq s\leq n+1-r \}&=&
 	\{(r,s), 1\leq r\leq n+1-s,2\leq s\leq n \},\\
 	\{(r-s,s), 2\leq r\leq n,2\leq s\leq r \}&=&
 		\{(r,s), 0\leq r\leq n-s,2\leq s\leq n \}, \end{eqnarray*}	$S_1+T_1=\sum_{2\leq s\leq n}Q(1,s)$. In the same way, one can see easily that{\footnotesize
 	\begin{eqnarray*}
 	S-S_1+T-T_1&=&-\sum_{\begin{array}{c} 2\leq k\leq\left[\frac{n+1}{2}\right]\\ l_1+\ldots+l_k=s\leq n+1
 	\end{array}	}\\&&\sum_{r=1}^{n+1-s}[A_{r,1}^0(x,y),
 		A_{k,1}^0(B_{l_{1}}(x),\ldots,B_{l_{k}}(x),A_{n+1-s-r,1}^0(x,z))]
 		-\sum_{r=0}^{n-s}
 		[A_{k,1}^0(B_{l_{1}}(x),\ldots,B_{l_{k}}(x),A_{r,1}^0(x,y)),
 		A_{n+1-s-r,1}^0(x,z)].
 		\end{eqnarray*}	}
 		
 		To conclude, we must show that $Q_1=Q_2$ where {
 			\begin{eqnarray*}
 				Q_1&=&
 				-\sum_{\substack{ 2\leq k\leq\left[\frac{n+1}{2}\right]\\ l_1+\ldots+l_k=s\leq n+1\\
 					}}\frac1{k!}\sum_{r=0}^{n+1-s}
 					\sum_{p=1}^{k-1}\sum_{\sigma\in S_k}\\
 					&&	[A_{p,1}^0(B_{l_{\sigma(1)}}(x),\ldots,B_{l_{\sigma(p)}}(x),A_{r,1}^0(x,y)),
 					A_{k-p,1}^0(B_{l_{\sigma(p+1)}}(x),\ldots,B_{l_{\sigma(k)}}(x),A_{n+1-s-r,1}^0(x,z))].
 				\end{eqnarray*}
 				}
 				
 				\begin{eqnarray*}
 					Q_2&=&-\sum_{r=2}^{n-1}\sum_{\begin{array}{c}1\leq p\leq\left[\frac{r}2\right]\\  s_1=l_1+\ldots+l_{p}\leq r\\\end{array}	}
 					\sum_{\begin{array}{c} 1\leq h\leq\left[\frac{n+1-r}2\right]\\ s_2=m_1+\ldots+m_{h}\leq n+1-r
 						\end{array}}\\&&[A_{p,1}^0(B_{l_1}(x),\ldots,B_{l_p}(x),A_{r-s_1,1}^0(x,y)),A_{h,1}^0(B_{m_1}(x),\ldots,B_{m_h}(x),A_{n+1-r-s_2,1}^0(x,z))].		\end{eqnarray*}
 					Denote by $\N_2=\{2,3,\ldots\}$,  for any $l=(l_1,\ldots,l_k)\in\N^k$, $|l|=l_1+\ldots+l_k$, for any $\sigma\in S_k$, $l^\sigma=(l_{\sigma(1)},\ldots,l_{\sigma(k)})$
 					and for $k\in\left\{1,\ldots,\left[\frac{n+1}2\right]\right\}$ and $s\leq n+1$ put
 					\[ \mathcal{S}(k,s)=\left\{ (l,\sigma,r,p)\in\N_2^k\times S_k\times \N\times\N,|l|=s, 0\leq r\leq n+1-s, 1\leq p\leq k-1 \right\} \]and for any
 					$(l,\sigma,r,p)\in \mathcal{S}(k,s)$ put
 					\[ \Phi(l,\sigma,r,p)=[A_{p,1}^0(B_{l_{\sigma(1)}}(x),\ldots,B_{l_{\sigma(p)}}(x),A_{r,1}^0(x,y)),
 					A_{k-p,1}^0(B_{l_{\sigma(p+1)}}(x),\ldots,B_{l_{\sigma(k)}}(x),A_{n+1-s-r,1}^0(x,z))]. \]Thus
 					\[ Q_1=-\sum_{\begin{array}{c} 2\leq k\leq\left[\frac{n+1}{2}\right],s\leq n+1
 						\end{array}	}\left[\frac1{k!}\sum_{(l,\sigma,r,p)\in \mathcal{S}(k,s) }\Phi(l,\sigma,r,p)\right]. \]

 					The map $S_k\times \mathcal{S}(k,s)\too \mathcal{S}(k,s)$, $(\mu,(l,\sigma,r,p))\mapsto (l^\mu,\sigma\circ\mu^{-1},r,p)$ defines a free action of $S_k$ and the map $[l,\sigma,r,p]\mapsto (l^{\sigma},r,p)$ identifies the quotient $\wi{\mathcal{S}(k,s)}$  to
 					\[ \wi{\mathcal{S}(k,s)}=\left\{ (l,r,p)\in\N_2^k\times \N\times\N,|l|=s, 0\leq r\leq n+1-s, 1\leq p\leq k-1 \right\}. \] Moreover, $\Phi(l,\sigma,r,p)=\Phi(l^\mu,\sigma\circ\mu^{-1},r,p)$ so
 					\[ Q_1=-\sum_{\begin{array}{c} 2\leq k\leq\left[\frac{n+1}{2}\right],s\leq n+1
 						\end{array}	}\left[\sum_{(l,r,p)\in \wi{\mathcal{S}(k,s)} }\Phi(l,\mathrm{Id},r,p)\right]. \]

 					On the other hand,  put
 					\[ T=\left\{ (l,m,p,q,r)\in \N_2^p\times\N_2^q\times\N\times\N\times\N,2\leq r\leq n-1,|l|\leq r,|m|\leq n+1-r,1\leq p\leq\left[\frac{r}2\right],1\leq q\leq\left[\frac{n+1-r}2\right]    \right\}. \]
 					We have
 					\[ Q_2=\sum_{(l,m,p,q,r)\in T}\Psi(l,m,p,q,r), \]where
 					\[ \Psi(l,m,p,q,r)=[A_{p,1}^0(B_{l_1}(x),\ldots,B_{l_p}(x),A_{r-|l|,1}^0(x,y)),A_{h,1}^0(B_{m_1}(x),\ldots,B_{m_h}(x),A_{n+1-r-|m|,1}^0(x,z))]. \]
 					We consider now 
 					\[ J:T\too\bigcup_{\begin{array}{c} 2\leq k\leq\left[\frac{n+1}{2}\right],s\leq n+1
 						\end{array}}\wi{\mathcal{S}(k,s)} \]given by
 					\[ J(l,m,p,q,r)=((l,m),r-|l|,p)\in \wi{\mathcal{S}(p+q,|l|+|m|)}.  \]Indeed, it is obvious that $2\leq p+q\leq \left[\frac{n+1}{2}\right]$, $1\leq p\leq p+q-1$. Moreover, since $|l|\leq r$ and $|m|\leq n+1-r$ then $0\leq r-|l|\leq n+1-(|l|+|m|)$ and hence $((l,m),r-|l|,p)\in \wi{\mathcal{S}(p+q,|l|+|m|)}.$.

 					$J$ is a bijection since we have
 					\[ J^{-1}(l,p,r)=((l_1,\ldots,l_p),(l_{p+1},\ldots,l_k),p,k-p,s=r+l_1+\ldots+l_p)\in T. \]Indeed, we have
 					\[ 2\leq k\leq\left[\frac{n+1}2\right],\; 1\leq p\leq k-1\esp 0\leq r\leq n+1-|l|.  \]This implies obviously that $1\leq p$ and $1\leq k-p$. Now $2p\leq l_1+\ldots+l_p$ and hence $p\leq\left[\frac{s}2\right]$. This with $k\leq\left[\frac{n+1}2\right]$ imply that $k-p\leq \left[\frac{n+1-s}2\right]$. It is obvious that $s\geq2$ and from
 					$0\leq r\leq n+1-|l|$, we get
 					\[ s\leq n+1-(l_{p+1}+\ldots l_k)\leq n-1\esp l_{p+1}+\ldots+l_{k}\leq n+1-s. \]
 					This completes the proof.
 										\end{proof}

 \section{Analytic linear Lie rack structures on $\mathfrak{sl}_2(\R)$ and $\mathfrak{so}(3)$ }\label{section5}
 
 We denote by $\mathfrak{sl}_2(\R)$ the Lie algebra of traceless real $2\times 2$-matrices and by $\mathfrak{so}(3)$ the Lie algebra of skew-symmetric real $3\times 3$-matrices. We consider them as left Leibniz algebras and the purpose of this section is to prove that they are rigid in the sense of Definition \ref{def}. Namely, we will  prove the following theorem. 
  \begin{theo}\label{h} Let $\h$ be either $\mathfrak{sl}_2(\R)$ or $\mathfrak{so}(3)$ and $\rhd$ an analytic linear Lie rack structure  on $\h$ such that $\br_\rhd$ is the Lie algebra bracket of $\h$. Then there exists an analytic function $F:\R\too\R$ given by
  	\[ F(u)=1+\sum_{k=1}^{\infty}a_ku^k \]such that, for any $x,y\in\h$,
  	\[ x\rhd y=\exp(F(\langle x,x\rangle )\ad_x)(y), \]where $\langle x,x\rangle=\frac12\tr(\ad_x\circ\ad_x)$. So $\h$ is rigid.
  	
  \end{theo}
  
  The proof of this theorem is based on Theorem \ref{main2}. So the first step is the  determination of symmetric invariant multilinear maps on $\h=\mathfrak{sl}_2(\R)$ and $\mathfrak{so}(3)$.
  To achieve that, we use  the Chevalley restriction theorem for vector-valued functions proved in \cite{KNV}. We recall its statement as explained in \cite{bala}.
  
  Let $\G$ be a complex semi-simple Lie algebra, $\h$ a Cartan subalgebra, $G$ the connected and simply connected Lie group of $\G$ and $H$ the maximal torus in $G$ generated by $\exp_G(\h)$. We denote by $N_G(H)$ the normalizer of $H$ in $G$. Note that for any $a\in N_G(H)$, $\Ad_a$ leaves $\h$ invariant and $W=\{\Ad_{a}{|\h},a\in N_G(H)   \}$ is the Weyl group of $\h$.
  Let $B:\G\times\ldots\times\G\too \G$ be a symmetric $n$-multilinear map which is $\G$-invariant, i.e., for any $a\in G$ and any $x_1,\ldots,x_n\in\G$,
  \begin{equation}\label{eqs1}
  B(\Ad_ax_1,\ldots,\Ad_ax_n)=\Ad_aB(x_1,\ldots,x_n).
  \end{equation}This is equivalent to
  \begin{equation}\label{eqs2}
  [y,B(x_1,\ldots,x_n)]=\sum_{i=1}^nB(x_1,\ldots,[y,x_i],\ldots,x_n),\quad y,x_1,\ldots,x_n\in\G.
  \end{equation}
  We denote by $S_n^{\G}(\G,\G)$ the vector space of $\G$-invariant $n$-multilinear symmetric forms on $\G$ with values in $\G$. 
  
  Let $B\in S_n^{\G}(\G,\G)$ and we denote by $\wi B$ its restriction to $\h$. From \eqref{eqs2}, we get that for any $y,x_1,\ldots,x_n\in\h$,
  \[ [y,\wi B(x_1,\ldots,x_n)]=0 \]and hence $\wi B(x_1,\ldots,x_n)\in\h$ (since $\h$ is a maximal abelian subalgebra). So $\wi B$ defines a $n$-multilinear symmetric map $\wi B:\h\times\ldots\times\h\too\h$ which is $W$-invariant. If we denote by $S_n^{W}(\h,\h)$ the vector space of $G$-invariant $n$-multilinear symmetric forms on $\h$ with values in $\h$, we get
  a map $\operatorname{Res}:S_n^{\G}(\G,\G)\too S_n^{W}(\h,\h)$. 
  \begin{theo}[\cite{KNV}] \label{thenkv} $\operatorname{Res}$ is injective.
  \end{theo}
  
  Let $\G$ be a real semi-simple Lie algebra. The definition of $S_n^{\G}(\G,\G)$ is similar to the complex case.
  The complexified  Lie algebra $\G^{\C}=\G\oplus i\G$ of $\G$ is also semi-simple and we have an injective map $S_n^{\G}(\G,\G)\too
  S_n^{\G^{\C}}(\G^{\C},\G^{\C})$, which assigns to each $\G$-invariant $n$-multilinear invariant form  $B$ on $\G$   the unique $\C$-multilinear map $B^\C$ from $\G^\C\times\ldots\times\G^\C$ to $\G^\C$ whose restriction to $\G$ is $B$. By using \eqref{eqs2} one can see easily that since $B$ is $\G$-invariant then $B^\C$ is $\G^\C$-invariant.

  We will now apply Theorem \ref{thenkv} and the embedding above to compute $S_n^{\G}(\G,\G)$ for any $n\in \N^*$ when $\G=\mathfrak{sl}_2(\C)$, $\G=\mathfrak{sl}_2(\R)$ or $\G=\mathfrak{so}(3)$.
  
  Let $\G=\mathfrak{sl}_2(\C)$, $\G=\mathfrak{sl}_2(\R)$ or $\G=\mathfrak{so}(3)$. For any $n\in\N^*$, we define
  $P:\G^{2n}\too\mathbb{K}$ $(\mathbb{K}=\R,\C)$ by
  \[ P_n(x_1,\ldots,x_{2n})=\frac1{(2n)!}\sum_{\sigma\in S_{2n}}\langle x_{\sigma(1)},x_{\sigma(2)}\rangle\ldots
  \langle x_{\sigma(2n-1)},x_{\sigma(2n)}\rangle\esp  P_0=1, \]where $\langle x,x\rangle=\frac12\tr(\ad_x^2)$.
  This defines a symmetric invariant form on $\G$ and the map $B_n^\G:\G^{2n+1}\too\G$ given by
  \[ B_n^\G(x_1,\ldots,x_{2n+1})=\sum_{k=1}^{2n+1}P_n(x_1,\ldots,\hat{x}_k,\ldots,x_{2n+1})x_k  \] is symmetric and invariant.

  \begin{theo} Let $\G=\mathfrak{sl}_2(\C)$.  Then, for any $n\in\N^*$, we have
  	\[ S_{2n}^{\G}(\G,\G)=0\esp S_{2n+1}^{\G}(\G,\G)=\C B_{n}^\G. \]
  	
  \end{theo}
  \begin{proof} A Cartan subalgebra of $\G$ is $\h=\C\left(\begin{matrix}
  	1&0\\0&-1
  	\end{matrix} \right)$ which is one dimensional and hence, for any $n\in\N^*$, the dimension of $S_n^{W}(\h,\h)$ is less than or equal to $\leq1$. By virtue of Theorem \ref{thenkv} we get $\dim S_{n}^{\G}(\G,\G)\leq1$. Moreover, the associated Lie group to $\h$ is $H=\left\{ \left(\begin{matrix}
  	z&0\\0&z^{-1}
  	\end{matrix} \right),z\in\C^*  \right\}$ and one can see easily that $a=\left(\begin{matrix}
  	0&1\\-1&0
  	\end{matrix} \right)\in N_G(H)$ and hence $Ad_{a}|\h\in W$. Now
  	\[ Ad_{a}\left(\begin{matrix}
  	1&0\\0&-1
  	\end{matrix} \right)=a\left(\begin{matrix}
  	1&0\\0&-1
  	\end{matrix} \right)a^{-1}=-\left(\begin{matrix}
  	1&0\\0&-1
  	\end{matrix} \right). \]
  	Thus for any $B\in S_n^{W}(\h,\h)$, the invariance by $\Ad_a$ implies that, for any $x_1,\ldots,x_n\in\h$,
  	\[ (-1)^nB(x_1,\ldots,x_n)=-B(x_1,\ldots,x_n). \]So if $n$ is even then $B=0$ and hence $S_n^{\G}(\G,\G)=0$. If $n=2p+1$ is odd, the restriction theorem shows that $\dim S_{2p+1}^{\G}(\G,\G)\leq1$ and since $B_p^\G\in S_{2p+1}^{\G}(\G,\G)$ we get the result. 
  \end{proof}
  If $\G=\mathfrak{sl}_2(\R)$ or $\G=\mathfrak{so}(3)$ then $\G^\C$ is isomorphic to $\mathfrak{sl}_2(\C)$ and since the invariants of $\G$ are embedded in the invariants of $\G^\C$ we get the following  corollary.
  \begin{co}\label{sym} If $\G=\mathfrak{sl}_2(\R)$ or $\G=\mathfrak{so}(3)$ then, for any $n\in\N^*$, we have
  	\[ S_{2n}^{\G}(\G,\G)=0\esp S_{2n+1}^{\G}(\G,\G)=\R B_{n}^\G. \]
  	
  \end{co}

  Let us pursue our preparation of the proof of Theorem \ref{h}. Let $\h=\mathfrak{sl}_2(\R)$ or $\mathfrak{so}(3,\R)$ and $x\in\h$. Then
  \[ x=\left( \begin{array}{cc}a&b\\c&-a\end{array}\right)\quad\mbox{or}\quad
  x=\left( \begin{array}{ccc}0&a&b\\-a&0&c\\-b&-c&0\end{array}\right). \]
 Put
 \[ \langle x,x\rangle=\left\{ \begin{array}{lcc}
 \frac12\tr(\ad_x^2)=2\tr(x^2)=4(a^2+bc)&\quad \mbox{if}\quad&\h=\mathfrak{sl}_2(\R),\\
 \frac12\tr(\ad_x^2)=\frac12\tr(x^2)=-a^2-b^2-c^2&\quad\mbox{if}\quad&\h=\mathfrak{so}(3).
 \end{array}     \right. \]The following formula which is true in both $\mathfrak{sl}_2(\R)$ and $\mathfrak{so}(3)$ is easy to check and will play a crucial role in the proof of Theorem \ref{h}. Indeed, for any $x,y\in\h$, 
 \begin{equation}\label{magic} \ad_x\circ\ad_x(z)=-\langle x,z\rangle x+\langle x,x\rangle z. \end{equation}
 This implies easily, by virtue of Corollary \ref{co1},  that 
 \begin{equation} \label{magic1}\left\{  \begin{array}{l} \di
 A_{2n,1}^0(x,y)=\frac{\langle x,x\rangle^{n-1}}{(2n)!}\ad_x^2(y)=\frac{\langle x,x\rangle^{n}}{(2n)!}y-\frac{\langle x,x\rangle^{n-1}\langle x,y\rangle}{(2n)!}x,\; n\geq1,\\\di
  A_{2n+1,1}^0(x,y)=\frac{\langle x,x\rangle^n}{(2n+1)!}[x,y],\quad n\geq0.\end{array}\right. \end{equation}

 \begin{pr}\label{pr} Let $\h$ be either $\mathfrak{sl}_2(\R)$ or $\mathfrak{so}(3)$ and $\rhd$ an analytic linear Lie rack product on $\h$ such that $\br_\rhd$ is the Lie algebra bracket of $\h$. Then there exists a sequence $(U_n)_{n\in\N^*}$ with $U_1=1$, $U_2=\frac12$, for any $x,y\in\h$,
 	\[ x\rhd y=y+  \left(\sum_{n=0}^\infty U_{2n+1}\langle x,x\rangle^n\right)[x,y]
 	+\left(\sum_{n=1}^\infty {U_{2n}\langle x,x\rangle^{n-1}}\right)\ad_x^2(y)
 	\] and for any $n\in\N^*$,
 	\[ U_{2n}=\frac12
 	\left[\sum_{r=0}^{n-1}{U_{2r+1}U_{2(n-r)-1}}-
 	\sum_{r=1}^{n-1}{U_{2r}U_{2(n-r)}}\right]. \]

 \end{pr}
 
 \begin{proof} By virtue of Theorem \ref{main},  $\di x\rhd y=\sum_{n=0}^\infty A_{n,1}(x,y)$ where the sequence $(A_{n,1})$ satisfies \eqref{eqm}. 
 	Moreover, since $\h$ is simple $H^0(\h)=H^1(\h)=0$ and we can apply  Theorem \ref{main2}. Thus 
 	$$\label{eq8} A_{n,1}(x,y)=A_{n,1}^0(x,y)+\sum_{\begin{array}{c} 2k\leq s=l_1+\ldots+l_{k}\leq n\\1\leq k\leq\left[\frac{n}2\right]
 		\end{array}	} A_{k,1}^0(B_{l_1}(x),\ldots,B_{l_k}(x),A_{n-s,1}^0(x,y)), $$and the $B_l$ are symmetric invariant. By virtue of Corollary \ref{sym},
 	\[ B_{2l}=0\esp B_{2l+1}(x)=c_l\langle x,x\rangle^lx. \]
 	Thus
 	\begin{eqnarray*}
 		A_{n,1}(x,y)&=&A_{n,1}^0(x,y)+\sum_{\begin{array}{c} 2k\leq s=2l_1+\ldots+2l_{k}+k\leq n\\1\leq k\leq\left[\frac{n}2\right]
 			\end{array}	}c_{l_1}\ldots c_{l_k} \langle x,x\rangle^{l_1+\ldots+l_k}A_{k,1}^0(x,A_{n-s,1}^0(x,y)).
 		\end{eqnarray*}But by using Corollary \ref{co1}, we have $\di A_{k,1}^0(x,A_{n-s,1}^0(x,y))=
 		\frac{(n+k-s)!}{k!(n-s)!}A_{n-2(l_1+\ldots+l_k)}^0(x,y)$ and hence we can write
 		\[ A_{n,1}(x,y)=\sum_{l=0}^{\left[\frac{n-1}2 \right]}K_{n,l}\langle x,x\rangle^lA_{n-2l}^0(x,y), \]where $K_{n,l}$ are constant such that $K_{n,0}=1$. Note that in particular $A_{2,1}(x,y)=A_{2,1}^0(x,y)$.
 		 		Now by using \eqref{magic1}, we get
 		\begin{eqnarray*}
 			A_{2n+1,1}(x,y)&=&\sum_{l=0}^{n} K_{2n+1,l}\langle x,x\rangle^lA_{2(n-l)+1}^0(x,y)\\
 			&=&\sum_{l=0}^{n} K_{2n+1,l}\langle x,x\rangle^l\frac{1}{(2(n-l)+1)!}\langle x,x\rangle^{n-l}[x,y]\\
 			&=&U_{2n+1}\langle x,x\rangle^n[x,y]={C_{2n+1}}A_{2n+1,1}^0(x,y),
 		\end{eqnarray*}where $\di U_{2n+1}=\frac{C_{2n+1}}{(2n+1)!}$ are constant. 
 		In the same way, one can show that there exists constants $\di U_{2n}=\frac{C_{2n}}{(2n)!}$ such that 		\[ A_{2n,1}(x,y)={U_{2n}\langle x,x\rangle^{n-1}}\ad_x^2(y)={C_{2n}}A_{2n,1}^0(x,y),  \]and get the desired expression of $x\rhd y$. 
 		
 		On the other hand, The equation \eqref{eqm} for $q=1$ and $p=2n$ holds for both the $A_n$ and the $A_n^0$ so we get
 		\begin{eqnarray*}
 		A_{2n,1}(x,[y,z])&=&C_{2n}A_{2n,1}^0(x,[y,z])\\&=&
 		C_{2n}[y,A_{2n,1}^0(x,z)]+C_{2n}[A_{{2n},1}^0(x,y),z]+C_{2n}\sum_{r=1}^{2n-1}[A_{r,1}^0(x,y),
 		A_{2n-r,1}^0(x,z)]\\
 		&=&C_{2n}[y,A_{2n,1}^0(x,z)]+C_{2n}[A_{2n,1}^0(x,y),z]+\sum_{r=1}^{2n-1}C_rC_{2n-r}[A_{r,1}^0(x,y),
 		A_{2n-r,1}^0(x,z)].
 		\end{eqnarray*}Thus
 		\[ \sum_{r=1}^{2n-1}(C_{2n}-C_rC_{2n-r})[A_{r,1}^0(x,y),
 		A_{2n-r,1}^0(x,z)]=0 \]and hence
 		\[ 0=\sum_{r=1}^{n-1}(C_{2n}-C_{2r}C_{2(n-r)})[A_{2r,1}^0(x,y),
 		A_{2(n-r),1}^0(x,z)]+\sum_{r=0}^{n-1}(C_{2n}-C_{2r+1}C_{2(n-r)-1})[A_{2r+1,1}^0(x,y),
 		A_{2(n-r-1)+1,1}^0(x,z)]. \]By using \eqref{magic1} we get
 		\begin{eqnarray*} 0&=&\sum_{r=1}^{n-1}\frac{(C_{2n}-C_{2r}C_{2(n-r)})\langle x,x\rangle^{n-2}}{(2r)!(2(n-r)!)}[\ad_x^2(y),
 			\ad_x^2(z)]+\sum_{r=0}^{n-1}\frac{(C_{2n}-C_{2r+1}C_{2(n-r)-1})\langle x,x\rangle^{n-1}}{(2r+1)!(2(n-r)-1)!}[[x,y],
 			[x,y]]. \end{eqnarray*}One can show easily by using \eqref{magic} that
 		\[  [\ad_x^2(y),
 		\ad_x^2(z)]+\langle x,x\rangle[[x,y],
 		[x,z]]=0\] and deduce that
 		\[ \sum_{r=1}^{n-1}\frac{1}{(2r)!(2(n-r)!)}(C_{2n}-C_{2r}C_{2(n-r)})=\sum_{r=0}^{n-1}\frac{1}{(2r+1)!(2(n-r)-1)!}(C_{2n}-C_{2r+1}C_{2(n-r)-1}). \]On the other hand
 		\[ 0=(1-1)^{2n}=\sum_{r=0}^{n}\frac{(2n)!}{(2r)!(2(n-r))!}- 
 		\sum_{r=0}^{n-1}\frac{(2n)!}{(2r+1)!(2(n-r)-1)!},\]
 	and finally,
 		\[ \frac{C_{2n}}{(2n)!}=\frac{1}{2}\left[\sum_{r=0}^{n-1}\frac{1}{(2r+1)!(2(n-r)-1)!}C_{2r+1}C_{2(n-r)-1}-\sum_{r=1}^{n-1}\frac{1}{(2r)!(2(n-r)!)}C_{2r}C_{2(n-r)}\right]. \]To complete the proof, it suffices to replace $\di\frac{C_{r}}{r!}$ by $U_r$.
 		\end{proof}
 
 \noindent{\bf Proof of Theorem \ref{h}}. \begin{proof} According to Proposition \ref{pr},
 	 there exists a sequence $(U_n)_{n\in\N^*}$ with $U_1=1$, $U_2=\frac12$, for any $x,y\in\h$,
 	\[ x\rhd y=y+  \left(\sum_{n=0}^\infty U_{2n+1}\langle x,x\rangle^n\right)[x,y]
 	+\left(\sum_{n=1}^\infty {U_{2n}\langle x,x\rangle^{n-1}}\right)\ad_x^2(y)
 	\] and for any $n\in\N^*$,
 	\begin{equation}\label{evend} U_{2n}=\frac12
 	\left[\sum_{r=0}^{n-1}{U_{2r+1}U_{2(n-r)-1}}-
 	\sum_{r=1}^{n-1}{U_{2r}U_{2(n-r)}}\right]. \end{equation}
 We will show that there exists a unique sequence $(a_n)_{n\geq1}$ such that the function 	 $F(t)=1+\sum_{t=1}^\infty a_nt^n$ satisfies
 	\[x\rhd y= \exp(F(\langle x,x\rangle )\ad_x)(y)=y+\sum_{n=0}^\infty F(\langle x,x\rangle)^{2n+1}A_{2n+1,1}^0(x,y)+\sum_{n=1}^\infty F(\langle x,x\rangle)^{2n}A_{2n,1}^0(x,y). \]Thus{\small
 		\begin{eqnarray*}
 			\exp(F(\langle x,x\rangle )\ad_x)(y)&=&y+
 			\left(\sum_{n=0}^\infty\frac{[F(\langle x,x\rangle)]^{2n+1}\langle x,x\rangle^n}{(2n+1)!}\right)[x,y]
 			+\left(\sum_{n=1}^\infty\frac{[F(\langle x,x\rangle)]^{2n}\langle x,x\rangle^{n-1}}{(2n)!}\right)\ad_x^2(y).
 		\end{eqnarray*}}Put
 		$[F(\langle x,x\rangle)]^{n}=\sum_{m=0}^\infty B_{n,m}\langle x,x\rangle^m$ and compute
 		  the coefficients $B_{n,m}$. Indeed,
 		\begin{eqnarray*} [F(\langle x,x\rangle)]^n&=&\left(1+a_1\langle x,x\rangle +a_2\langle x,x\rangle^2+\ldots+ a_m\langle x,x\rangle^m+R\right)^{n}\\
 			&=&
 			\left(1+a_1\langle x,x\rangle +a_2\langle x,x\rangle^2+\ldots+ a_m\langle x,x\rangle^m\right)^{n}+P, \end{eqnarray*}where $P$ contains terms of degree $\geq m+1$. The multinomial theorem gives
 		\[ \left(1+a_1\langle x,x\rangle +a_2\langle x,x\rangle^2+\ldots+ a_m\langle x,x\rangle^m\right)^{n}=
 		\sum_{k_0+\ldots+k_{m}=n}\frac{n!}{k_0!k_1!\ldots k_m!}a_1^{k_1}\ldots a_m^{k_m}\langle x,x\rangle^{k_1+2k_2+\ldots+mk_m}. \]
 		Thus
 		\[ B_{n,0}=1\esp B_{n,m}=\sum_{k_1+2k_2+\ldots+mk_{m}=m,k_0+k_1+\ldots+k_m= n}\frac{n!}{k_0!k_1!\ldots k_m!}a_1^{k_1}\ldots a_m^{k_m}\;\mbox{for}\; m\geq1. \]

 		So
 		\begin{eqnarray*}
 			\sum_{n=0}^\infty\frac{F(\langle x,x\rangle)^{2n+1}\langle x,x\rangle^n}{(2n+1)!}&=&\sum_{n=0}^\infty\sum_{m=0}^\infty \frac{B_{2n+1,m}\langle x,x\rangle^{m+n}}{(2n+1)!}
 			=\sum_{n=0}^{\infty}\left( \sum_{p=0}^{n}\frac{B_{2p+1,n-p}}{(2p+1)!}\right) \langle x,x\rangle^{n},\\
 		\sum_{n=1}^\infty\frac{F(\langle x,x\rangle)^{2n}\langle x,x\rangle^{n-1}}{(2n)!}&=&\sum_{n=1}^\infty\sum_{m=0}^\infty \frac{B_{2n,m}\langle x,x\rangle^{m+n-1}}{(2n)!}
 		=\sum_{n=1}^{\infty}\left( \sum_{p=1}^{n}\frac{B_{2p,n-p}}{(2p)!}\right) \langle x,x\rangle^{n-1}.	
 		\end{eqnarray*}For sake of simplicity and clarity, put
 		\[ V_{n,m}(a_1,\ldots,a_m)=\frac{B_{n,m}}{n!}
 		=\sum_{k_1+2k_2+\ldots+mk_{m}=m,k_0+k_1+\ldots+k_m= n}\frac{a_1^{k_1}\ldots a_m^{k_m}}{k_0!k_1!\ldots k_m!}. \]
 		
 		To prove the theorem we need to show that there exists a unique sequence $(a_n)_{n\geq1}$ such that
 		\begin{eqnarray}\label{od}
 		U_{2n+1}&=&\sum_{p=0}^{n}V_{2p+1,n-p}(a_1,\ldots,a_{n-p}),\quad n\geq1,\\
 		\label{even}U_{2n}&=&\sum_{p=1}^{n}V_{2p,n-p}(a_1,\ldots,a_{n-p}),\quad n\geq1.
 		\end{eqnarray}
 		Note first that the relation \eqref{evend} and the fact that $U_2=\frac12$ defines the sequence $(U_{2n})_{n\geq1}$ entirely in function of the sequence $(U_{2n+1})_{n\geq0}$. On the other hand, since $V_{1,n}(a_1,\ldots,a_n)=a_n$ and $U_1=1$ then
 		\[ U_3=a_1+\frac1{3!}\esp U_{2n+1}=a_n+\sum_{p=1}^{n}V_{2p+1,n-p}(a_1,\ldots,a_{n-p}),\quad n\geq2. \]
 		Since the quantity $\sum_{p=1}^{n}V_{2p+1,n-p}(a_1,\ldots,a_{n-p})$ depends only on $(a_1,\ldots,a_{n-1})$, these relations  define inductively and uniquely the sequence $(a_n)_{n\geq1}$ in function of $(U_{2n+1})_{n\geq0}$. To achieve the proof we need to prove \eqref{even}. We will proceed by induction and we will use the following relation
 		\begin{equation}\label{magic3} \frac{\partial V_{n,m}}{\partial a_l}(a_1,\ldots,a_m)=V_{n-1,m-l}(a_1,\ldots,a_{m-l}),\quad l=1,\ldots,m. \end{equation}
 		Indeed,
 		\begin{eqnarray*}
 		\frac{\partial V_{n,m}}{\partial a_l}(a_1,\ldots,a_m)&=&
 		\sum_{k_1+2k_2+\ldots+mk_{m}=m,k_0+k_1+\ldots+k_m= n,k_l\geq1}\frac{a_1^{k_1}\ldots a_l^{k_l-1}\ldots a_m^{k_m}}{k_0!k_1!\ldots(k_l-1)!\ldots k_m!}\\
 		&\stackrel{k_l'=k_l-1}=&\sum_{k_1+2k_2+\ldots+ lk_l'+\ldots+mk_{m}=m-l,k_0+k_1+\ldots+ k_l'+\ldots+k_m= n-1}\frac{a_1^{k_1}\ldots a_l^{k_l'}\ldots a_m^{k_m}}{k_0!k_1!\ldots(k_l')!\ldots k_m!}.
 		\end{eqnarray*}To conclude, it suffices to remark that in the relation
 		\[ k_1+2k_2+\ldots +lk_l'+\ldots+mk_{m}=m-l \]
 		the left side is a sum of nonnegative numbers and the right side is nonnegative so $(m-l+1)k_{m-l+1}=\ldots=mk_m=0$ and hence the relation is equivalent to
 		\[ k_1+2k_2+\ldots +(m-l)k_{m-l}=m-l. \]
 		Now, we are able to prove \eqref{even}. We proceed by induction. For $n=1$, we have $U_2=\frac12$ and $V_{2,0}=\frac12$. Suppose that the relation holds from 1 to $n-1$. By virtue of \eqref{evend}, we have
 	$$U_{2n}=\frac12
 	\left[\sum_{r=0}^{n-1}{U_{2r+1}U_{2(n-r)-1}}-
 	\sum_{r=1}^{n-1}{U_{2r}U_{2(n-r)}}\right]$$and all the $U_r$ appearing in this formula are given by \eqref{od} and \eqref{even} this implies that  $U_{2n}$ is a function of $(a_1,\ldots,a_{n-1})$ and we can put $U_{2n}=H(a_1,\ldots,a_{n-1})$. We can also put 
 	\[ \sum_{p=1}^nV_{2p,n-p}(a_1,\ldots,a_{n-p})=G(a_1,\ldots,a_{n-1}). \]
 	To show that $U_{2n}$ satisfies \eqref{even} is equivalent to showing
 	\[ H(0)=G(0)\esp \frac{\partial H}{\partial a_l}=\frac{\partial G}{\partial a_l},\quad l=1,\ldots n-1. \]
 	But $V_{n,m}(0)=0$ if $m\geq1$ and $V_{n,0}(0)=\frac1{n!}$. Hence
 	\begin{eqnarray*}
 	H(0)&=&\frac12\left(\sum_{r=0}^{n-1} \frac{1}{(2r+1)!(2(n-r)-1)!}
 	-\sum_{r=1}^{n-1} \frac{1}{(2r)!(2(n-r))!}  \right)\\
 	&=&\frac12\left(\sum_{r=0}^{n-1} \frac{1}{(2r+1)!(2(n-r)-1)!}
 	-\sum_{r=0}^{n} \frac{1}{(2r)!(2(n-r))!}  \right)+\frac{1}{(2n)!}\\
 	&=&-\frac12(1-1)^{2n}+\frac{1}{(2n)!}=\frac{1}{(2n)!},\\
 	G(0)&=&V_{2n,0}(0)=\frac{1}{(2n)!}=H(0).
 	\end{eqnarray*}For $r=0,\ldots,n-1$, by induction hypothesis $U_{2r+1}$ is given by \eqref{od} and by using \eqref{magic3} one can see easily that
 	$\di\frac{\partial U_{2r+1}}{\partial a_l}=U_{2(r-l)}$ if $l=1,\ldots,r$ and 0 if $l\geq r+1$. Similarly, we have $\di\frac{\partial U_{2r}}{\partial a_l}=U_{2(r-l)-1}$ if $l=1,\ldots,r-1$ and 0 if $l\geq r$. For sake of simplicity, we put
 	\[ \di\frac{\partial U_{2r+1}}{\partial a_l}=U_{2(r-l)}\esp \di\frac{\partial U_{2r}}{\partial a_l}=U_{2(r-l)-1}  \]with the convention $U_0=1$ and $U_s=0$ if $s$ is negative. Then, for $l=1,\ldots,n-1$, we have
 	\begin{eqnarray*}
 	\frac{\partial H}{\partial a_l}&=&\frac12
 	\left[\sum_{r=0}^{n-1}\left(\frac{\partial U_{2r+1}}{\partial a_l}U_{2(n-r)-1}
 	+\frac{\partial U_{2(n-r)-1}}{\partial a_l}U_{2r+1}\right)-
 	\sum_{r=1}^{n-1}\left(\frac{\partial U_{2r}}{\partial a_l}U_{2(n-r)}
 	+\frac{\partial U_{2(n-r)}}{\partial a_l}U_{2r}\right)\right]\\
 	&=&\frac12
 	\left[\sum_{r=0}^{n-1}\left({ U_{2(r-l)}}U_{2(n-r)-1}
 	+{ U_{2(n-r-l-1)}}U_{2r+1}\right)-
 	\sum_{r=1}^{n-1}\left({ U_{2(r-l)-1}}U_{2(n-r)}
 	+{ U_{2(n-r-l)-1}}U_{2r}\right)\right]\\
 	&=&\frac12\sum_{r=0}^{n-1-l}{ U_{2r}}U_{2(n-r-l)-1}+\frac12\sum_{r=0}^{n-1}{ U_{2(n-r-l-1)}}U_{2r+1}-\frac12\sum_{r=0}^{n-l-2}{ U_{2r+1}}U_{2(n-r-l-1)}
 	-\frac12\sum_{r=1}^{n-1}{ U_{2(n-r-l)-1}}U_{2r}\\
 	&=&\frac12U_{2(n-l)-1}+\frac12\sum_{r=n-l-1}^{n-1}{ U_{2(n-r-l-1)}}U_{2r+1}
 	-\frac12\sum_{r=n-l}^{n-1}{ U_{2(n-r-l)-1}}U_{2r}\\
 	&=&U_{2(n-l)-1}.
 	\end{eqnarray*}This completes the proof.
 	\end{proof}


\end{document}